\documentclass[11pt, one side,emlines]{amsart}
\usepackage[a4paper, margin=1in]{geometry}
\usepackage{amssymb,latexsym,xy,eucal,mathrsfs}
\textwidth=18 cm \textheight=27cm \theoremstyle{plain}
\usepackage{lineno}
\usepackage{lipsum}
\newtheorem{theorem}{Theorem}[section]
\newtheorem{maintheorem}{Main Theorem}[section]
\newtheorem{proposition}[theorem]{Proposition}
\newtheorem{lemma}[theorem]{Lemma}

\newtheorem{corollary}[theorem]{Corollary}
\newtheorem{definition}[theorem]{Definition}

\newtheorem{conjecture}{Conjecture}[section]

\begin{document}
\title{on conditional connectivity of  the Cartesian product of  cycles }
\author{ J. B. Saraf$^1$, Y. M. Borse$^2$ and Ganesh Mundhe$^3$}
\address{1. Amruteshwar Arts, Commerce and Science College, Vinzar- 412211, INDIA }
\address{2. Department of Mathematics, SPPU, Pune-411007, INDIA}
\address{3. Army Institute of Technology, Pune-411015, INDIA }
\email{1. sarafjb@gmail.com, 2. ymborse11@gmail.com, 3. ganumundhe@gmail.com}
\maketitle
\baselineskip18truept
\begin{abstract}
The conditional $h$-vertex($h$-edge) connectivity of a connected graph $H$ of minimum degree $ k > h$ is the size of a smallest vertex(edge) set $F$ of $H$ such that  $H - F$ is a disconnected graph of minimum degree at least $h.$ Let $G$ be the Cartesian product of $r\geq 1$ cycles, each of length at least four and let $h$ be an integer such that $0\leq h\leq 2r-2$. In this paper, we determine the conditional $h$-vertex-connectivity and the conditional $h$-edge-connectivity of the graph $G.$  We prove that both these connectivities are equal to  $(2r-h)a_h^r$, where $a_h^r$ is the number of vertices of a smallest $h$-regular subgraph of $G.$
\end{abstract}
\section{Introduction}
One of the feature of a good interconnection network is its high fault tolerance capacity. Interconnection network can be modelled into a graph with the help of which we can study many properties of the network. Connectivity of a modelled graph measures the fault tolerance capacity of the interconnection
network. High fault tolerance capacity of the network plays an important role in practice. Traditional
connectivities have some limitations to measure the fault tolerance capacity of a network accurately. In order
to compute traditional edge connectivity, one allows failure of all the links incident with the same processor,
practically which is rare. This affect the reliability of the network. One can overcome these limitations effectively by considering the conditional
connectivity of  graphs introduced by Harary \cite{harary}.

Let $G$ be a connected graph with minimum degree at least $k\geq 1$ and let $h$ be an integer such that
$0\leq h <k $. A set $F$ of edges(vertices) of $G$ such that $G-F$ is disconnected and each component of it has minimum
degree at least $h$ is an \textit{$h$-edge(vertex) cut} of $G$. The \textit{conditional $h$-edge(vertex)} connectivity of $G,$ denoted by $\lambda^h(G)$ ($\kappa^h(G)$), is the minimum cardinality $|F|$ of an $h$-edge(vertex) cut $F$ of $G.$  Clearly, $h = 0$ gives the traditional edge(vertex) connectivity.

  Many researcher   have worked on the  problem of determining  the conditional connectivities for various classes of graphs and determined these parameters for smaller values of $h$ \cite{esfa 89, esfa 88,xu,  latifi}.  Exact values of one or both conditional connectivities are known for some classes of graphs.  For the $n$-dimensional hypercube $Q_n$, the conditional connectivities $\lambda^h$ and $\kappa^h$ are same and their common value is $2^h(n-h);$ see \cite{duksu, xu}.  Li and Xu \cite{ li1} proved that  $\lambda^h$  of any $n$-dimensional hypercube-like network $G_n$ is also  $2^h(n-h).$   Ye and Liang \cite{ye} obtained a lower bound on the conditional $h$-vertex connectivity $\kappa^h$ of the graph $G_n$ and  established that $\kappa^h$ is $2^h(n-h)$ for  some  members of hypercube-like networks such as Crossed cubes, Locally twisted cubes, M\"{o}bius cubes. Independently,  Wei and Hsieh \cite{wei} determined $\kappa^h$ for the Locally twisted cubes.
Ning \cite{ning} obtained $\kappa^h$ for the exchanged crossed cubes. Both  $\lambda^h$ and  $\kappa^h$ are determined for the class of $(n,~k)$-star graphs by Li et al. \cite{li2}.

An $r$-dimensional torus is the Cartesian product of $r$ cycles. The \textit{ $k$-ary $r$-cube,} denoted by  $Q^k_r,$ is the Cartesian product of $r$ cycles each of length $k.$  In particular, the hypercube $Q_{2r}$ is  $Q^4_r.$   The multidimensional torus, the $k$-ary $r$-cube and the hypercube are widely used interconnection  networks; see \cite{ch, le, lw, xxh}.

It is easy to see that an  $r$-dimensional torus
is a $2r$-regular graph with traditional vertex connectivity and edge connectivity  $2r;$ see \cite{xxh}.  In this paper, we determine the conditional $h$-edge-connectivity as well as the conditional $h$-vertex-connectivity of the given multidimensional torus.

By $ C_k$ we mean a cycle of length $k.$ For integers $h,r,k_1,k_2,\dots, k_r$ with $0\leq h\leq 2r$ and $4\leq k_1\leq k_2\leq \dots \leq k_r,$ we define a quantity $a_h^r$ as follows,
\begin{definition}\label{ahr}
\[a_h^r= \left\{ \begin{array}{ll}
           2^h & \mbox{if $0\leq h \leq r$}\\\\
         2^{r-i}~k_1k_2\cdots k_i & \mbox{if $h=r+i,$~ $1\leq i \leq r $}.\end{array} \right. \] 

\end{definition}
 
We prove that both the conditional connectivities $\lambda^h$ and $k^h$ are equal to $a_h^r(2r-h)$ for the Cartesian product of cycles  $ C_{k_1}, C_{k_2}, \dots, C_{k_r}.$ 

The following is the main theorem of the paper.
\begin{maintheorem}\label{main theorem}
Let $h,~r,~k_1,~k_2,\dots, k_r$ be integers such that $0\leq h \leq 2r-2 $ and $4\leq k_1 \leq k_2\leq \dots \leq k_r$ and let  $ G $ be the Cartesian product of the cycles $ C_{k_1}, C_{k_2}, \dots, C_{k_r}.$  Then $\lambda^h(G)=k^h(G) = a_h^r(2r-h).$ 
\end{maintheorem}
\begin{corollary}
Let $h,~r, ~k$ be integers such that $0\leq h \leq 2r-2,$ $ 4 \leq k$ and let $ Q_r^k $ be the $k$-ary $r$-cube. Then $\lambda^h(Q_r^k)=k^h(Q_r^k) = a_h^r(2r-h),$ where $a_h^r= 2^h$ if $0\leq h \leq r$ and $a_h^r=2^{r-i}k^i$ if  $h=r+i$ and $1\leq i \leq r.$
\end{corollary}
\begin{corollary} [\cite {duksu, xu}]
 For integers  $h$ and $r$ with $0\leq h \leq 2r-2 ,$  $\lambda^h(Q_{2r})=k^h(Q_{2r}) = 2^h(2r-h).$
\end{corollary}

 

In Section 2, we prove that a smallest $h$-regular subgraph of the graph $G$ of the above theorem has $a_h^r$ vertices. We also derive some of its properties. In Section 3, we obtain vertex connectivity and in Section 4, we determine the conditional edge connectivity of  $G$. 
  
  \section{Smallest $h$-regular subgraph}
  In this section, we define a smallest $h$-regular subgraph of the Cartesian product of $r$-cycles and obtain some properties of it.
  
   The \textit{Cartesian product} of two graphs $G$ and $H$ is a graph $ G \Box H$ with vertex set $ V(G)\times V(H).$ Two vertices $(x, y)$ and $ (u, v)$ are adjacent in $ G \Box H$ if and only if either $ x = u$ and $ y $ is adjacent to $v$ in $H,$ or $ y = v$ and $ x $ is adjacent to $u$ in $G.$  The hypercube $Q_n$ is the Cartesian product of $n$ copies of the complete graph $K_2.$

   {\bf Notation:} Consider the graph $G$ of Main Theorem \ref{main theorem}. We have $G = C_{k_1}\Box C_{k_2}\Box \cdots \Box C_{k_r},$ where $C_{k_i}$ is a cycle of length $k_i$ for $i=1,2,\cdots,k_r$ and $4 \leq k_1\leq k_2\leq \dots \leq k_r.$  We can write $G$ as $G=H\Box C_{k_r},$ where $H=C_{k_1}\Box C_{k_2}\Box \cdots \Box C_{k_{r-1}}.$  Label by $1, 2, \dots , k_r$ the vertices of the cycle $C_{k_r}$ so that $i$ is adjacent to $(i+1)\pmod {k_r}.$ Hence $G$ can be obtained by replacing $i^{th}$ vertex  of $C_{k_r}$ by a copy $H^i$ of $H$ and replacing edge joining $i$ and $i+1$ of $C_{k_r}$ by the perfect matching $M_i$ between the corresponding vertices of $H^i$ and $H^{i+1}.$ Thus $G= H^1\cup H^2\cup \dots \cup H^{k_r}\cup (M_1\cup M_2 \cup \dots \cup M_{k_r})$ (see Figure 1).        \begin{center}
      \unitlength 1mm 
      \linethickness{0.4pt}
      \ifx\plotpoint\undefined\newsavebox{\plotpoint}\fi 
      \begin{picture}(74.528,57.278)(0,0)
      \put(-.722,6.528){\framebox(9.5,26.25)[cc]{}}
      \put(31.778,6.528){\framebox(9.5,26.25)[cc]{}}
      \put(14.778,6.528){\framebox(9.5,26.25)[cc]{}}
      \put(47.278,6.528){\framebox(9.5,26.25)[cc]{}}
      \put(65.028,6.528){\framebox(9.5,26.25)[cc]{}}
      \put(8.778,28.278){\line(1,0){6}}
      \put(41.278,28.278){\line(1,0){6}}
      \put(9.028,23.778){\line(1,0){5.75}}
      \put(41.528,23.778){\line(1,0){5.75}}
      \put(8.778,13.778){\line(1,0){6}}
      \put(41.278,13.778){\line(1,0){6}}
      \put(12.028,20.028){\makebox(0,0)[cc]{$\vdots$}}
      \put(28.636,19.818){\makebox(0,0)[cc]{$\vdots$}}
      \put(60.801,19.818){\makebox(0,0)[cc]{$\vdots$}}
      \put(44.528,20.028){\makebox(0,0)[cc]{$\vdots$}}
      \qbezier(.028,32.778)(34.028,57.278)(73.028,32.778)
      \qbezier(3.028,33.028)(32.903,50.528)(69.278,33.028)
      \qbezier(8.278,32.778)(35.278,40.903)(65.278,32.528)
      \put(36.352,39.318){\makebox(0,0)[cc]{$\vdots$}}
     
      \put(24.736,28.31){\line(1,0){.9724}}
      \put(26.681,28.31){\line(1,0){.9724}}
      \put(28.626,28.31){\line(1,0){.9724}}
      \put(30.57,28.31){\line(1,0){.9724}}
      \put(24.736,23.685){\line(1,0){.9724}}
      \put(26.681,23.685){\line(1,0){.9724}}
      \put(28.626,23.685){\line(1,0){.9724}}
      \put(30.57,23.685){\line(1,0){.9724}}
      \put(24.736,14.225){\line(1,0){.9724}}
      \put(26.681,14.225){\line(1,0){.9724}}
      \put(28.626,14.225){\line(1,0){.9724}}
      \put(30.57,14.225){\line(1,0){.9724}}
      \put(57.531,28.31){\line(1,0){.9724}}
      \put(59.476,28.31){\line(1,0){.9724}}
      \put(61.421,28.31){\line(1,0){.9724}}
      \put(63.365,28.31){\line(1,0){.9724}}
      \put(57.531,23.685){\line(1,0){.9724}}
      \put(59.476,23.685){\line(1,0){.9724}}
      \put(61.421,23.685){\line(1,0){.9724}}
      \put(63.365,23.685){\line(1,0){.9724}}
      \put(57.531,14.225){\line(1,0){.9724}}
      \put(59.476,14.225){\line(1,0){.9724}}
      \put(61.421,14.225){\line(1,0){.9724}}
      \put(63.365,14.225){\line(1,0){.9724}}
       \put(3.278,3.528){\makebox(0,0)[cc]{$H^1$}}
            \put(19.028,3.778){\makebox(0,0)[cc]{$H^2$}}
            \put(36.028,3.528){\makebox(0,0)[cc]{$H^t$}}
            \put(51.278,3.528){\makebox(0,0)[cc]{$H^{t+1}$}}
            \put(69.528,4.028){\makebox(0,0)[cc]{$H^{k_r}$}}
            \put(11.278,10.278){\makebox(0,0)[cc]{$M_1$}}
            \put(37.062,47.568){\makebox(0,0)[cc]{$M_{k_r}$}}
            \put(44.278,10.528){\makebox(0,0)[cc]{$M_t$}}
      \put(36.579,-3){\makebox(0,0)[cc]{Figure 1. $G=H \Box C_{k_r}$}}
      \end{picture}
      
        \end{center}
   \vskip.5cm
   From the following lemma, it is clear that $G$ is a $2r$-regular and $2r$-connedcted graph on $k_1k_2\dots k_r$ vertices. 
   \begin{lemma}\label{regular} 
   	If $G_i$ is a $m_i$-regular and $m_i$-connected graph on $n_i$ vertices for $i=1,2$, then $G_1\Box G_2$ is an $(m_1+m_2)$-regular and  $(m_1+m_2)$-connected graph on $n_1n_2$ vertices.
   \end{lemma}

    
  \textbf { Henceforth, by the graph $G$ we mean the graph $C_{k_1}\Box C_{k_2}\Box \cdots \Box C_{k_{r}}$ with $ 4 \leq k_1 \leq k_2\leq \dots \leq k_{r}.$}

   We now define a  $h$-regular subgraph, denoted by $W_h^r,$ of the graph $G$  as follows. 
   \begin{definition}\label{whr} For $4 \leq k_1\leq k_2\leq \dots \leq k_r$ and $0\leq h \leq 2r$, let
     \[W_h^r= \left\{ \begin{array}{ll}
              Q_h & \mbox{if $0\leq h \leq r$}\\\\
            Q_{r-i}\Box C_{k_1}\Box \textbf{}C_{k_2} \Box \cdots\Box C_{k_i} & \mbox{if $h=r+i$~ and ~ $1\leq i \leq r $}.\end{array} \right .\] 
   \end{definition}
    
 In the following two figures, a 2-regular subgraph $W_2^2$   and a 3-regular subgraph $W_3^2$  of the graph $C_5 \Box C_5$ are shown by bold lines. 
  \begin{center}
  \unitlength 1mm 
  \linethickness{0.4pt}
  \ifx\plotpoint\undefined\newsavebox{\plotpoint}\fi 
  \begin{picture}(85.562,44.147)(0,0)
  \put(.335,39.818){\circle*{1}}
  \put(53.732,39.187){\circle*{1}}
  \put(6.011,39.857){\circle*{1}}
  \put(59.408,39.226){\circle*{1}}
  \put(.335,34.338){\circle*{1}}
  \put(53.732,33.707){\circle*{1}}
  \put(5.801,34.128){\circle*{1}}
  \put(59.198,33.497){\circle*{1}}
  \put(11.477,34.128){\circle*{1}}
  \put(64.874,33.497){\circle*{1}}
  \put(16.943,34.128){\circle*{1}}
  \put(70.34,33.497){\circle*{1}}
  \put(22.619,34.128){\circle*{1}}
  \put(76.016,33.497){\circle*{1}}
  \put(27.664,34.128){\circle*{1}}
  \put(81.061,33.497){\circle*{1}}
  \put(11.702,39.857){\circle*{1}}
  \put(65.099,39.226){\circle*{1}}
  \put(.335,28.858){\circle*{1}}
  \put(53.732,28.227){\circle*{1}}
  \put(5.801,28.648){\circle*{1}}
  \put(59.198,28.017){\circle*{1}}
  \put(11.477,28.648){\circle*{1}}
  \put(64.874,28.017){\circle*{1}}
  \put(16.943,28.648){\circle*{1}}
  \put(70.34,28.017){\circle*{1}}
  \put(22.619,28.648){\circle*{1}}
  \put(76.016,28.017){\circle*{1}}
  \put(27.664,28.648){\circle*{1}}
  \put(81.061,28.017){\circle*{1}}
  \put(.335,23.378){\circle*{1}}
  \put(53.732,22.747){\circle*{1}}
  \put(5.801,23.168){\circle*{1}}
  \put(59.198,22.537){\circle*{1}}
  \put(11.477,23.168){\circle*{1}}
  \put(64.874,22.537){\circle*{1}}
  \put(16.943,23.168){\circle*{1}}
  \put(70.34,22.537){\circle*{1}}
  \put(22.619,23.168){\circle*{1}}
  \put(76.016,22.537){\circle*{1}}
  \put(27.664,23.168){\circle*{1}}
  \put(81.061,22.537){\circle*{1}}
  \put(17.196,39.857){\circle*{1}}
  \put(70.593,39.226){\circle*{1}}
  \put(.335,17.898){\circle*{1}}
  \put(53.732,17.267){\circle*{1}}
  \put(5.801,17.688){\circle*{1}}
  \put(59.198,17.057){\circle*{1}}
  \put(11.477,17.688){\circle*{1}}
  \put(64.874,17.057){\circle*{1}}
  \put(16.943,17.688){\circle*{1}}
  \put(70.34,17.057){\circle*{1}}
  \put(22.619,17.688){\circle*{1}}
  \put(76.016,17.057){\circle*{1}}
  \put(27.664,17.688){\circle*{1}}
  \put(81.061,17.057){\circle*{1}}
  \put(22.676,39.857){\circle*{1}}
  \put(76.073,39.226){\circle*{1}}
  \put(.335,12.838){\circle*{1}}
  \put(53.732,12.207){\circle*{1}}
  \put(5.801,12.628){\circle*{1}}
  \put(59.198,11.997){\circle*{1}}
  \put(11.477,12.628){\circle*{1}}
  \put(64.874,11.997){\circle*{1}}
  \put(16.943,12.628){\circle*{1}}
  \put(70.34,11.997){\circle*{1}}
  \put(22.619,12.628){\circle*{1}}
  \put(76.016,11.997){\circle*{1}}
  \put(27.664,12.628){\circle*{1}}
  \put(81.061,11.997){\circle*{1}}
  \put(27.735,39.857){\circle*{1}}
  \put(81.132,39.226){\circle*{1}}
  \put(5.792,39.428){\line(0,-1){5.303}}
  \put(11.468,39.428){\line(0,-1){5.303}}
  \put(64.865,38.797){\line(0,-1){5.303}}
  \put(16.934,39.428){\line(0,-1){5.303}}
  \put(70.331,38.797){\line(0,-1){5.303}}
  \put(22.61,39.428){\line(0,-1){5.303}}
  \put(76.007,38.797){\line(0,-1){5.303}}
  \put(27.655,39.428){\line(0,-1){5.303}}
  \put(81.052,38.797){\line(0,-1){5.303}}
  \put(6.402,39.848){\line(1,0){5.303}}
  \put(59.799,39.217){\line(1,0){5.303}}
  \put(6.191,12.939){\line(1,0){5.303}}
  \put(59.588,12.308){\line(1,0){5.303}}
  \put(6.191,18.195){\line(1,0){5.303}}
  \put(59.588,17.564){\line(1,0){5.303}}
  \put(6.402,23.45){\line(1,0){5.303}}
  \put(59.799,22.819){\line(1,0){5.303}}
  \put(5.981,28.916){\line(1,0){5.303}}
  \put(59.378,28.285){\line(1,0){5.303}}
  \put(5.981,34.172){\line(1,0){5.303}}
  \put(59.378,33.541){\line(1,0){5.303}}
  \put(11.867,39.848){\line(1,0){5.303}}
  \put(65.264,39.217){\line(1,0){5.303}}
  \put(11.657,12.939){\line(1,0){5.303}}
  \put(65.054,12.308){\line(1,0){5.303}}
  \put(11.657,18.195){\line(1,0){5.303}}
  \put(65.054,17.564){\line(1,0){5.303}}
  \put(11.867,23.45){\line(1,0){5.303}}
  \put(65.264,22.819){\line(1,0){5.303}}
  \put(11.447,28.916){\line(1,0){5.303}}
  \put(64.844,28.285){\line(1,0){5.303}}
  \put(11.447,34.172){\line(1,0){5.303}}
  \put(64.844,33.541){\line(1,0){5.303}}
  \put(.726,39.848){\line(1,0){5.303}}
  \put(.726,23.45){\line(1,0){5.303}}
  \put(.515,12.939){\line(1,0){5.303}}
  \put(.515,18.195){\line(1,0){5.303}}
  \put(.326,17.541){\line(0,-1){4.596}}
  \put(5.581,28.052){\line(0,-1){4.596}}
  \put(.115,28.473){\line(0,-1){4.596}}
  \put(5.792,17.331){\line(0,-1){4.596}}
  \put(11.468,17.331){\line(0,-1){4.596}}
  \put(64.865,16.7){\line(0,-1){4.596}}
  \put(16.934,17.331){\line(0,-1){4.596}}
  \put(70.331,16.7){\line(0,-1){4.596}}
  \put(22.61,17.331){\line(0,-1){4.596}}
  \put(76.007,16.7){\line(0,-1){4.596}}
  \put(27.655,17.331){\line(0,-1){4.596}}
  \put(81.052,16.7){\line(0,-1){4.596}}
  \put(23.033,39.848){\line(1,0){4.596}}
  \put(76.43,39.217){\line(1,0){4.596}}
  \put(22.823,12.939){\line(1,0){4.596}}
  \put(76.22,12.308){\line(1,0){4.596}}
  \put(22.823,18.195){\line(1,0){4.596}}
  \put(76.22,17.564){\line(1,0){4.596}}
  \put(23.033,23.45){\line(1,0){4.596}}
  \put(76.43,22.819){\line(1,0){4.596}}
  \put(22.612,28.916){\line(1,0){4.596}}
  \put(76.009,28.285){\line(1,0){4.596}}
  \put(22.612,34.172){\line(1,0){4.596}}
  \put(76.009,33.541){\line(1,0){4.596}}
  \put(.255,22.951){\line(0,-1){.8838}}
  \put(.255,21.183){\line(0,-1){.8838}}
  \put(.255,19.415){\line(0,-1){.8838}}
  \put(53.653,22.32){\line(0,-1){.8838}}
  \put(53.653,20.552){\line(0,-1){.8838}}
  \put(53.653,18.785){\line(0,-1){.8838}}
  \put(5.721,22.741){\line(0,-1){.8838}}
  \put(5.721,20.973){\line(0,-1){.8838}}
  \put(5.721,19.205){\line(0,-1){.8838}}
  \put(59.119,22.11){\line(0,-1){.8838}}
  \put(59.119,20.342){\line(0,-1){.8838}}
  \put(59.119,18.575){\line(0,-1){.8838}}
  \put(11.397,22.741){\line(0,-1){.8838}}
  \put(11.397,20.973){\line(0,-1){.8838}}
  \put(11.397,19.205){\line(0,-1){.8838}}
  \put(64.795,22.11){\line(0,-1){.8838}}
  \put(64.795,20.342){\line(0,-1){.8838}}
  \put(64.795,18.575){\line(0,-1){.8838}}
  \put(16.863,22.741){\line(0,-1){.8838}}
  \put(16.863,20.973){\line(0,-1){.8838}}
  \put(16.863,19.205){\line(0,-1){.8838}}
  \put(70.261,22.11){\line(0,-1){.8838}}
  \put(70.261,20.342){\line(0,-1){.8838}}
  \put(70.261,18.575){\line(0,-1){.8838}}
  \put(22.539,22.741){\line(0,-1){.8838}}
  \put(22.539,20.973){\line(0,-1){.8838}}
  \put(22.539,19.205){\line(0,-1){.8838}}
  \put(75.937,22.11){\line(0,-1){.8838}}
  \put(75.937,20.342){\line(0,-1){.8838}}
  \put(75.937,18.575){\line(0,-1){.8838}}
  \put(27.584,22.741){\line(0,-1){.8838}}
  \put(27.584,20.973){\line(0,-1){.8838}}
  \put(27.584,19.205){\line(0,-1){.8838}}
  \put(80.982,22.11){\line(0,-1){.8838}}
  \put(80.982,20.342){\line(0,-1){.8838}}
  \put(80.982,18.575){\line(0,-1){.8838}}
  \put(17.482,39.778){\line(1,0){.8838}}
  \put(19.25,39.778){\line(1,0){.8838}}
  \put(21.018,39.778){\line(1,0){.8838}}
  \put(70.88,39.147){\line(1,0){.8838}}
  \put(72.647,39.147){\line(1,0){.8838}}
  \put(74.415,39.147){\line(1,0){.8838}}
  \put(17.272,12.869){\line(1,0){.8838}}
  \put(19.04,12.869){\line(1,0){.8838}}
  \put(20.808,12.869){\line(1,0){.8838}}
  \put(70.67,12.238){\line(1,0){.8838}}
  \put(72.437,12.238){\line(1,0){.8838}}
  \put(74.205,12.238){\line(1,0){.8838}}
  \put(17.272,18.125){\line(1,0){.8838}}
  \put(19.04,18.125){\line(1,0){.8838}}
  \put(20.808,18.125){\line(1,0){.8838}}
  \put(70.67,17.494){\line(1,0){.8838}}
  \put(72.437,17.494){\line(1,0){.8838}}
  \put(74.205,17.494){\line(1,0){.8838}}
  \put(17.482,23.38){\line(1,0){.8838}}
  \put(19.25,23.38){\line(1,0){.8838}}
  \put(21.018,23.38){\line(1,0){.8838}}
  \put(70.88,22.749){\line(1,0){.8838}}
  \put(72.647,22.749){\line(1,0){.8838}}
  \put(74.415,22.749){\line(1,0){.8838}}
  \put(17.061,28.846){\line(1,0){.8838}}
  \put(18.829,28.846){\line(1,0){.8838}}
  \put(20.597,28.846){\line(1,0){.8838}}
  \put(70.459,28.215){\line(1,0){.8838}}
  \put(72.226,28.215){\line(1,0){.8838}}
  \put(73.994,28.215){\line(1,0){.8838}}
  \put(17.061,34.102){\line(1,0){.8838}}
  \put(18.829,34.102){\line(1,0){.8838}}
  \put(20.597,34.102){\line(1,0){.8838}}
  \put(70.459,33.471){\line(1,0){.8838}}
  \put(72.226,33.471){\line(1,0){.8838}}
  \put(73.994,33.471){\line(1,0){.8838}}
  \put(11.379,28.998){\line(0,-1){5.48}}
  \put(64.776,28.367){\line(0,-1){5.48}}
  \put(11.379,34.253){\line(0,-1){5.48}}
  \put(64.776,33.622){\line(0,-1){5.48}}
  \put(16.845,28.998){\line(0,-1){5.48}}
  \put(70.242,28.367){\line(0,-1){5.48}}
  \put(16.845,34.253){\line(0,-1){5.48}}
  \put(70.242,33.622){\line(0,-1){5.48}}
  \put(22.521,28.998){\line(0,-1){5.48}}
  \put(75.918,28.367){\line(0,-1){5.48}}
  \put(22.521,34.253){\line(0,-1){5.48}}
  \put(75.918,33.622){\line(0,-1){5.48}}
  \put(27.566,28.998){\line(0,-1){5.48}}
  \put(80.963,28.367){\line(0,-1){5.48}}
  \put(27.566,34.253){\line(0,-1){5.48}}
  \put(80.963,33.622){\line(0,-1){5.48}}
  \put(-2.383,23.475){\line(0,1){0}}
  \put(51.015,22.844){\line(0,1){0}}
  \put(16.957,37.35){\line(0,1){0}}
  \put(70.355,36.719){\line(0,1){0}}
  \thicklines
  \put(5.887,34.477){\line(0,1){0}}
  \put(59.284,33.846){\line(0,1){0}}
  \multiput(5.887,34.477)(-.030143,-.090143){7}{\line(0,-1){.090143}}
  \multiput(59.284,33.846)(-.030143,-.090143){7}{\line(0,-1){.090143}}
  \put(.421,29.221){\line(0,1){0}}
  \put(53.818,28.59){\line(0,1){0}}
  \multiput(.421,29.221)(-.030143,-.06){7}{\line(0,-1){.06}}
  \multiput(53.818,28.59)(-.030143,-.06){7}{\line(0,-1){.06}}
  \put(5.887,23.335){\line(0,1){0}}
  \put(59.284,22.704){\line(0,1){0}}
  \multiput(5.887,23.335)(.0323077,.0323846){13}{\line(0,1){.0323846}}
  \multiput(59.284,22.704)(.0323077,.0323846){13}{\line(0,1){.0323846}}
  \thinlines
  \qbezier(.21,40.153)(14.085,44.147)(27.96,39.733)
  \qbezier(53.607,39.522)(67.482,43.516)(81.357,39.102)
  \qbezier(27.751,12.614)(13.876,8.62)(.001,13.034)
  \qbezier(81.148,11.983)(67.273,7.989)(53.398,12.403)
  \qbezier(.21,29.221)(14.085,33.216)(27.96,28.801)
  \qbezier(53.607,28.59)(67.482,32.585)(81.357,28.17)
  \qbezier(.42,34.477)(14.295,38.472)(28.17,34.057)
  \qbezier(53.817,33.846)(67.692,37.841)(81.567,33.426)
  \qbezier(.421,23.756)(14.296,27.75)(28.17,23.335)
  \qbezier(53.818,23.125)(67.693,27.119)(81.567,22.704)
  \qbezier(.21,18.29)(14.085,22.284)(27.96,17.869)
  \qbezier(53.607,17.659)(67.482,21.653)(81.357,17.238)
  \qbezier(11.352,39.733)(8.304,26.068)(11.563,13.244)
  \qbezier(64.749,39.102)(61.701,25.437)(64.96,12.613)
  \qbezier(17.028,39.733)(13.98,26.068)(17.239,13.244)
  \qbezier(70.425,39.102)(67.377,25.437)(70.636,12.613)
  \qbezier(22.494,39.733)(19.446,26.068)(22.705,13.244)
  \qbezier(75.891,39.102)(72.843,25.437)(76.102,12.613)
  \put(.421,23.545){\line(0,1){0}}
  \put(53.818,22.914){\line(0,1){0}}
  \thicklines
  \put(5.676,23.125){\line(0,1){0}}
  \put(59.073,22.494){\line(0,1){0}}
  \thinlines
  \multiput(.421,40.153)(-.030143,-.810857){7}{\line(0,-1){.810857}}
  \thicklines
  \put(5.676,23.335){\line(0,1){0}}
  \put(59.073,22.704){\line(0,1){0}}
  \put(-0,23.756){\line(0,1){0}}
  \put(53.397,23.125){\line(0,1){0}}
  \put(6.307,25.017){\line(0,1){0}}
  \put(59.704,24.386){\line(0,1){0}}
  \put(5.887,28.801){\line(0,1){0}}
  \put(59.284,28.17){\line(0,1){0}}
  \put(5.887,28.801){\line(1,0){.631}}
  \put(59.284,28.17){\line(1,0){.631}}
  \put(.421,34.057){\line(1,0){5.676}}
  \put(53.818,17.449){\line(1,0){5.676}}
  \put(53.818,39.102){\line(1,0){5.676}}
  \put(.42,29.012){\line(0,1){5.676}}
  \put(53.817,12.404){\line(0,1){5.676}}
  \put(54.027,27.961){\line(0,1){5.676}}
  \put(54.027,22.915){\line(0,1){5.676}}
  \put(53.817,34.057){\line(0,1){5.676}}
  \put(5.886,29.011){\line(0,1){5.676}}
  \put(59.283,12.403){\line(0,1){5.676}}
  \put(59.493,27.96){\line(0,1){5.676}}
  \put(59.493,22.914){\line(0,1){5.676}}
  \put(59.283,34.056){\line(0,1){5.676}}
  \put(.631,28.801){\line(1,0){5.676}}
  \put(54.028,12.193){\line(1,0){5.676}}
  \put(53.818,28.17){\line(1,0){5.676}}
  \put(53.818,23.125){\line(1,0){5.676}}
  \put(54.028,33.846){\line(1,0){5.676}}
  \put(.631,34.057){\line(1,0){5.466}}
  \put(54.028,17.449){\line(1,0){5.466}}
  \put(54.028,39.102){\line(1,0){5.466}}
  \put(.21,29.432){\line(0,1){5.466}}
  \put(53.607,12.824){\line(0,1){5.466}}
  \put(53.817,28.381){\line(0,1){5.466}}
  \put(53.817,23.335){\line(0,1){5.466}}
  \put(53.607,34.477){\line(0,1){5.466}}
  \put(5.886,29.221){\line(0,1){5.466}}
  \put(59.283,12.613){\line(0,1){5.466}}
  \put(59.493,28.17){\line(0,1){5.466}}
  \put(59.493,23.124){\line(0,1){5.466}}
  \put(59.283,34.266){\line(0,1){5.466}}
  \put(.841,28.801){\line(1,0){5.466}}
  \put(54.238,12.193){\line(1,0){5.466}}
  \put(54.028,28.17){\line(1,0){5.466}}
  \put(54.028,23.125){\line(1,0){5.466}}
  \put(54.238,33.846){\line(1,0){5.466}}
  \put(.211,34.267){\line(1,0){5.886}}
  \put(53.608,17.659){\line(1,0){5.886}}
  \put(53.608,39.312){\line(1,0){5.886}}
  \put(.211,34.267){\line(1,0){5.676}}
  \put(53.608,17.659){\line(1,0){5.676}}
  \put(53.608,39.312){\line(1,0){5.676}}
  \put(.21,28.801){\line(0,1){5.676}}
  \put(53.607,12.193){\line(0,1){5.676}}
  \put(53.817,27.75){\line(0,1){5.676}}
  \put(53.817,22.704){\line(0,1){5.676}}
  \put(53.607,33.846){\line(0,1){5.676}}
  \put(5.676,28.801){\line(0,1){5.676}}
  \put(59.073,12.193){\line(0,1){5.676}}
  \put(59.283,27.75){\line(0,1){5.676}}
  \put(59.283,22.704){\line(0,1){5.676}}
  \put(59.073,33.846){\line(0,1){5.676}}
  \put(.421,29.011){\line(1,0){5.676}}
  \put(53.818,12.403){\line(1,0){5.676}}
  \put(53.608,28.38){\line(1,0){5.676}}
  \put(53.608,23.335){\line(1,0){5.676}}
  \put(53.818,34.056){\line(1,0){5.676}}
  \put(-0,34.057){\line(1,0){5.886}}
  \put(53.397,17.449){\line(1,0){5.886}}
  \put(53.397,39.102){\line(1,0){5.886}}
  \put(.421,28.591){\line(0,1){5.886}}
  \put(53.818,11.983){\line(0,1){5.886}}
  \put(54.028,27.54){\line(0,1){5.886}}
  \put(54.028,22.494){\line(0,1){5.886}}
  \put(53.818,33.636){\line(0,1){5.886}}
  \put(5.887,28.591){\line(0,1){5.886}}
  \put(59.284,11.983){\line(0,1){5.886}}
  \put(59.494,27.54){\line(0,1){5.886}}
  \put(59.494,22.494){\line(0,1){5.886}}
  \put(59.284,33.636){\line(0,1){5.886}}
  \put(.211,28.801){\line(1,0){5.886}}
  \put(53.608,12.193){\line(1,0){5.886}}
  \put(53.398,28.17){\line(1,0){5.886}}
  \put(53.398,23.125){\line(1,0){5.886}}
  \put(53.608,33.846){\line(1,0){5.886}}
  \thinlines
  \qbezier(27.54,40.153)(32.165,26.909)(27.54,12.824)
  \qbezier(80.937,39.522)(85.562,26.278)(80.937,12.193)
  \qbezier(6.097,39.943)(10.722,26.699)(6.097,12.614)
  \qbezier(59.704,39.733)(64.329,26.489)(59.704,12.404)
  \qbezier(59.704,39.523)(64.329,26.279)(59.704,12.194)
  \qbezier(59.494,39.312)(64.119,26.068)(59.494,11.983)
  \qbezier(-0,12.824)(-4.625,26.068)(-0,40.153)
  \qbezier(53.607,12.404)(48.982,25.648)(53.607,39.733)
  \qbezier(53.607,12.614)(48.982,25.858)(53.607,39.943)
  \qbezier(53.397,12.193)(48.772,25.437)(53.397,39.522)
  \put(12.824,3.784){\makebox(0,0)[cc]{(a).   $W_2^2$ }}
  \put(66.852,4.835){\makebox(0,0)[cc]{(b).   $W_3^2$}}
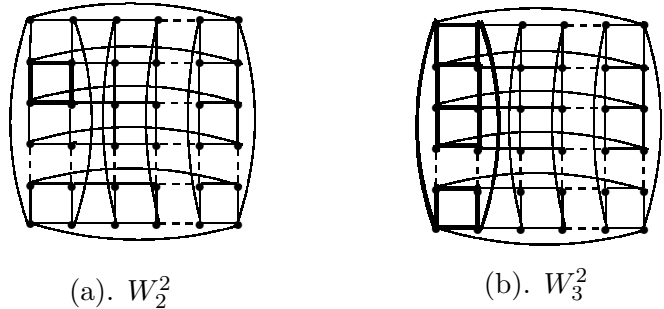
  \put(40.363,-1.892){\makebox(0,0)[cc]{Figure 2. The subgraph $W_2^2 $  and $W_3^2$ of $C_5\Box C_5$}}
  \end{picture}

  \end{center} 
   \vskip.3cm
   
   It is known that  a smallest $h$-regular subgraph of the hypercube $Q_n$ is isomorphic to $Q_h$ (see \cite{bs}). 
 We prove the analogous result for the Cartesian product of cycles. In fact, we establish that  $W_h^r$
is a smallest $h$-regular subgraph of the above graph $G.$   

The following lemma follows from Lemma \ref{regular}, definition of the number $a_h^r$ and the fact that the hypercube $Q_n$ is an $n$-regular, $n$-connected graph on $2^n$ vertices for any integer $ n \geq 0.$    
   \begin{lemma} \label{whr-reg}
   The graph $W_h^r$ is $h$-regular and $h$-connected with $a_h^r$ vertices.
   \end{lemma}
     	 
   We need the following lemma that gives relations between different values of $a_h^r.$
  \begin{lemma}\label{pro ahr} Let $ r \geq 2$ and let $ a_h^r$ be the quantity given in Definition \ref{ahr}. Then the following statements hold.
  	\begin{enumerate}  
  \item $a_h^r=2a_{h-1}^{r-1}$ if $1\leq h \leq 2r-1.$
 \item  $k_ra_{h-2}^{r-1}\geq a_h^r$ if $2\leq h \leq 2r.$
  \item  $a_h^{r-1}\geq a_h^r$ if $0\leq h \leq 2r-2.$ 
    \end{enumerate}
  \end{lemma}
  \begin{proof}
  Recall that $ a_h^r = 2^h$ if  $ 0 \leq h \leq r$ and  $ a_h^r = 2^{r-i}k_1k_2\dots k_i$ if  $ h = r + i $ with $ 1 \leq i \leq r,$ where $ 4 \geq k_1 \leq k_2 \leq \dots \leq K_{k_r}.$  
  
  (1). If $1\leq h \leq r,$ then  $a_h^r=2^h = 2 . 2^{h-1} = 2 a_{h-1}^{r-1}.$ For  $r+1\leq h\leq 2r-1,$ we have $h=r+i$ for some $1\leq i \leq r-1.$ Hence  $h-1=(r-1)+i$  gives   $a_{h-1}^{r-1}=2^{(r-1)-i}k_1k_2\dots k_i.$ Therefore $2a^{r-1}_{h-1}= a_h^r.$ 
  
  (2). Suppose $2\leq h \leq r+ 1.$ Then $a_{h-2}^{r-1}=2^{h-2},$ and $ a_h^r = 2^h$ if $ h < r + 1$ and  $ a_h^r = 2^{r-1}k_1$ if  $h=r+1.$  For  $r+2\leq h \leq 2r,$ we have  $h-2=(r-1)+(i-1) $ for some $2\leq i \leq r$ and so,   $a_{h-2}^{r-1}=2^{r-i}k_1k_2\dots k_{i-1}.$ Therefore,  $k_ra_{h-2}^{r-1}\geq a_h^r$ in each case as $k_r\geq k_i \geq k_1 \geq 4.$ 
  
  (3). Note that $ a_h^{r-1} = 2^h$  for  $1\leq h \leq r-1,$    and  $a_h^{r-1}=2^{(r-2)}k_1$  for  $h=r=(r-1)+1,$ and finally, $a_h^{r-1}=2^{r-i-2}k_1k_2\dots k_ik_{i+1}$ 
  for $h=(r-1)+(i+1)$ for $1\leq i \leq r.$ 
  Since $k_{i+1}\geq K_1 \geq  4,$ we have $a_h^{r-1}\geq a_h^r$ in all the three cases.
  \end{proof}
  
  \begin{lemma}\label{ahr 1}
 Every subgraph of the graph $G$ of minimum degree at least $h$ has at least $a_h^r$ vertices.
  \end{lemma}
  \begin{proof}
   We prove the result by induction on $r.$ The result holds obviously for $h=0$ and $h=1$ and so it holds for $r=1.$ Assume that $r\geq 2$ and $h\geq 2.$ We have $ G = C_{k_1}\Box C_{k_2}\dots C_{k_r},$ where $ 4 \leq k_1 \leq k_2 \leq \dots \leq K_{k_r}.$  Write $G$ as $G=H \Box C_{k_r},$  where $H=C_{k_1}\Box C_{k_2}\dots C_{k_{r-1}}.$ Then $G= H^1\cup H^2\cup \dots \cup H^{k_r}\cup (M_1\cup M_2 \cup \dots \cup M_{k_r}),$ where $H^i$ is the copy of $H$ corresponding to vertex $i$ of $C_{k_r}$ and $M_i$ is the perfect matching between the corresponding vertices of $H^i$ and $H^{i+1}.$ 
  
  Let $K$ be a subgraph of $G$ with minimum degree at least $h.$ Then  $K$  intersects at least one $H^i.$ Let $K_i=K\cap H^i$ for $i=1, 2, \dots k_r.$ We may assume that $K_1\neq \emptyset.$ 
  
  (1). Suppose $K=K_1.$ Then $K$ is contained in $H^1$ and the minimum degree of $K$ in $H^1$ is $h.$ Since $H^1$ is $2(r-1)$-regular, $h\leq 2r-2.$ Suppose  $h=2r-2.$ Then $K=H^1$ and so,  $|V(K)|=k_1k_2\dots k_{r-1}.$ If $ r = 2,$  then $|V(K)| = k_1 \geq 4 = a_2^2 = a_h^r.$ If $ r\geq 3,$ then   $ |V(K)|\geq 4k_1k_2\dots k_{r-2} =a_h^r$ as $k_{r-1}\geq 4.$   If $h<2r-2,$  then, by induction and Lemma \ref{pro ahr}(3), we have    $|V(K)|\geq a_h^{r-1}\geq a_h^r.$
   
  (2). Suppose $K_i\neq \emptyset$ for each $i.$ Note that in the graph $G,$ every vertex of $H^i$ has exactly one neighbour in $H^{i-1}$ and one in $H^{i+1}.$ Hence the minimum degree of $K^i$ is at least $h-2.$ By induction, $|V(K_i)|\geq a_{h-2}^{r-1}.$ Therefore, by Lemma \ref{pro ahr}(2),
  $$|V(K)|=|K_1|+|K_2|+\dots +|K_{k_r}|\geq k_ra_{h-2}^{r-1}\geq a_h^r.$$
    
 (3). Suppose $K_i\neq \emptyset$ for some $i.$ Let $t>1$ be the largest integer such that $K_t\neq \emptyset.$ Then the minimum degrees of $K_1$ and $K_t$ are at least $h-1.$ Hence, by induction $|V(K_1)|\geq a_{h-1}^{r-1}$ and $|V(K_t)|\geq a_{h-1}^{r-1}.$ Thus, by Lemma \ref{pro ahr}(1),
  $$|V(K)|\geq |V(K_1)|+|V(K_t)|\geq 2a_{h-1}^{r-1}=a_h^r.$$
  This completes the proof.
   \end{proof}

   As a consequence of Lemma \ref{whr-reg} and Lemma \ref{ahr 1}, we get the following result.
\begin{corollary} $W_h^r$ is a smallest subgraph of  the graph $G$ of minimum degree at least $h.$
\end{corollary} 
 
     \begin{lemma}\label{1 nb} If $ 0\leq h < 2r-1$  and $K$ is a subgraph of $ G$ isomorphic to the graph $W_h^r,$  then every vertex of  $G$ belonging to $V(G) - V(K)$ has at most one neighbour in the subgraph $K.$ 
    \end{lemma}
    \begin{proof}
    We proceed by induction on $r.$ If $ r = 1,$ then $G $ is just a cycle and so the result holds obviously. Suppose $r\geq 2.$ Assume that the result holds for the Cartesian product any $r-1$ cycles. We have $G=H\Box C_{k_r}.$ Then  $G= H^1\cup H^2\cup \dots \cup H^{k_r}\cup )M_1\cup M_2 \cup \dots \cup M_{k_r}),$ where $H^i$ is the copy of $H$ corresponding to vertex $i$ of $C_{k_r}$ and $M_i$ is the perfect matching between the corresponding vertices of $H^i$ and $H^{i+1}.$  Since the graph  $K$ is isomorphic to $W_h^r=W_{h-1}^{r-1}\Box K_2, $ $K$ is contained in two adjacent copies of $H^i.$  We may assume that $K$ is a subgraph of $H^2\cup H^3\cup M_2.$ Let $ K _i = K \cap H^i$ for $ i = 2, 3.$ Then $K_i$ is isomprphic to $W_{h-1}^{r-1}.$ Let $x$ be any vertex of $V(G)-V(K).$ If $x$ is in $V(H^2),$  then, by induction, $x$ has at most one neighbour in $K_2.$ Then $x$ has no neighbour in $K_3$ and so, it has at most one neighbour in $K.$ Similarly, $x$ has at most one neighbour in $K$ if it belongs to $V(H^3).$   Suppose $x$ is in $H^j$ for some $j \notin\{2, 3\}.$  Then $x$ has exactly one neighbour in $H^{j+1}$ and  one in $H^{j-1}$ each and no neighbour in $H^i$ for any $i\notin\{j-1, j+1\}.$ This shows that  $x$ has at most one neighbour in $H^2\cup H^3$ and hence in $K$  as  $k_r\geq 4.$ This completes the proof. 
    \end{proof}

 The above result is analogous to the result for hypercubes which states that if $K$ is a subgraph the hypercube $Q_n$ isomorphic to $Q_h$ with $ h < n,$ then every vertex of $Q_n$ which is not in $K$ has at most one neighbour in the graph $K;$ see \cite{bs}.

    We introduce few more notations which are required to prove some properties of the graph $W_h^r.$

   Let $S$ be a set  of vertices of a graph $K.$ A vertex of $K,$ which is not in $S,$ is a \textit{neighbour} of $S$ if it is adjacent to a vertex in $S.$ Denote by $N(S)$ the set of neighbours of $S.$ Let  $N[S] = N(S)\cup S.$ Sometimes we write $N_K(S)$ for $N(S)$ and $N_K[S]$ for $N[S]$ to specify the graph $K.$ The subgraph of $K$ induced by $S$ is denoted by $[S].$  
   
   \begin{lemma}\label{2 nb}
     If  $0\leq h \leq 2r-1$ and  $S=V(W_h^r),$  then any vertex of $G$ which is not in  $N[S]$ has at most two neighbours in the set $N[S].$  
     \end{lemma}
     
    \begin{proof} We proceed by induction on $r.$  The result holds trivially for $r=1.$ Suppose $ r\geq 2.$  Assume that the result holds for the Cartesian product of any $r-1$ cycles. Write $G$ as $H\Box C_{k_r}, $ where $H =C_{k_1}\Box C_{k_2}\Box \cdots \Box C_{k_{r-1}}.$   Since the graph $W_h^r$ is isomorphic to $W_{h-1}^{r-1}\Box K_2,$ we may assume that $W_h^r$ is a subgraph of $H \Box K_2$  by considering $W_{h-1}^{r-1}$ as a subgraph of $H.$  Hence, we may assume that $W_h^r$ is a subgraph of  $H^2 \cup H^3 \cup M_2,$ where $M_2$ is the perfect matching between $H^2$ and $H^3.$  Let $S_i = V(W_h^r \cap H^i)$ for $i=2,3.$ Then $ S = V(W_h^r) = S_2 \cup S_3.$

    	Let $ x \in V(G)-N[S].$  Then  $ x $ is a vertex of $H^j$
 for some $j.$ If $ j > 4,$ then, by the definition of the Cartesian product of graphs, $x$ has at most two neighbours in the set $V(H^1) \cup V(H^2)\cup V(H^3)\cup V(H^4)$  and so in  its subset $N[S].$ Suppose $ j \in \{1, 2,3,4\} .$ Then  $h < 2r-1$ as for $ h = 2r - 1,$  $W_h^r  = H^2 \cup H^3 \cup M_2$  and so, $N[S]= V(H^1)\cup V(H^2)\cup V(H^3)\cup V(H^4).$   Let $S_2'$ be the set of neighbours  of $S_2$ that are present in $H^1$ and let $S_3'$ be the set of neighbours of  $S_3$  in $H^4.$  Then  $N[S]= N_{H^2}[S_2]\cup N_{H^3}[S_3]\cup S_2' \cup S_3'.$ Further, the subgraph $[S_i']$ of $G$ induced by the set $S_i'$ is isomorphic to the graph $W_{h-1}^{r-1}$ for $ i = 2, 3.$    If $ j \in \{1, 4\},$  then, by Lemma \ref{1 nb},  $x$ has at most one neighbour in $S_2'\cup S_3'$ and at most one in $V(H^2) \cup V(H^3).$ Also, if $ j \in \{2, 3\},$   then, by induction, $x$ has at most two neighbours in $N_{H^i}[S_i]$ and no neighbour in $S_2' \cup S_3'.$ Thus, in any case, $x$ has at most two neighbour in $N[S].$  
   \end{proof}

We need  the following lemma to obtain an upper bound on $k^h(G)$.
\begin{lemma}\label{spanning}
     For $0\leq h \leq 2r-1,$ the inequality $(2r - h + 1) a_h^r \leq k_1k_2\dots k_r$  holds. Moreover, the inequality is strict if  $ h < 2r -1.$  
     \end{lemma}
     \begin{proof}
     Recall that $4\leq k_1\leq k_2\dots \leq k_r,$ and  $ a_h^r = 2^h$ if $ h \leq r$ and   $ a_h^r = 2^{(r-i)}k_1k_2\dots k_{i}$ if $ h = r+i.$  For convenience, let $L = (2r - h + 1) a_h^r$ and $ R = k_1k_2\dots k_r.$ Obviously,   $L = a_h^r = R $ if $ h = 2r,$ and $ L = 2a_h^r= 4 k_1k_2\dots k_{r-1} \leq R$ for $ h = 2r-1.$  Suppose $ h \leq 2r-2.$ If $ h = 0 $ or $ h = 1,$  then $L < 4^r \leq R.$ Similarly, if $ 2\leq h \leq r,$ then $L  < 2ra_h^r = 2r 2^h \leq 2r 2^r \leq 4^r \leq R$ as  $2r \leq 2^r.$       Suppose $ h = r + i $ with  $ 1\leq i \leq r-2.$  Then $L=(r-i+1)2^{r-i}k_1k_2\dots k_i < 2^{2(r-i)}k_1k_2\dots k_i$ as $2l \leq 2^l,$ if $l\geq 1$.  This shows that $L\leq 4^{r-i}k_1k_2\dots k_i\leq k_1k_2\dots k_r=R.$
          \end{proof}
   
     \section{Conditional Vertex Connectivity }

     Recall from Section 2 that $G$ denote the graph $C_{k_1}\Box C_{k_2} \Box \cdots \Box C_{k_r}$ and $W_h^r$ denote is a particular $h$-regular subgraph of  $G$ with $ a_h^r$ vertices,  where $4\leq k_1 \leq k_2\leq \dots \leq k_r$   In this section, we obtain the conditional vertex connectivity $k^h(G).$  For a graph $K$  and its subgraph $H,$ let $\delta(K)$ denote the minimum degree of $K$  while  $\delta_K(H)$ denote the minimum degree of $H$ in $K.$   

In the following lemma, we obtain an upper bound on $k^h(G).$
\begin{lemma}\label{vertex ub}
If $0\leq h \leq 2r-2,$ then  $k^h(G)\leq (2r-h)a_h^r.$ 
\end{lemma}
\begin{proof}
We have  $G=C_{k_1}\Box C_{k_2}\Box \cdots \Box C_{k_r}.$ Let $S$ be the set of vertices of the $h$-regular subgraph $W_h^r$ of $G.$ Then $|S|= a_h^r.$   By Lemma \ref{1 nb},  every vertex of $S$ has $2r - h$ neighbours in  $N(S).$  Hence $|N(S)| = (2r - h) |S| = (2r - h)  a_h^r.$   This gives $|N[S]| = |S \cup N(S)| = |S|+ |N(S)| = (2r - h + 1) a_h^r.$ Therefore, by Lemma \ref{spanning}, $|N[S]| < k_1k_2\dots k_r=|V(G)|.$ Hence $V(G) - N[S]$ is a non-empty set and by Lemma \ref{2 nb}, every member  of this set has at most two neighbours in $N[S].$ Consequently,  the minimum degree of the subgraph of $G$ induced by this set is at least $ 2r - 2 \geq h.$ Already, the minimum degree of the graph $[S] = W_h^r$ is $h.$ Hence the graph $ G - N(S)$ is disconnected and every component of it has minimum degree at least $h.$ Thus $ N(S)$ is a $h$-vertex cut of $G.$ Therefore $ \kappa^h(G)\leq |N(S)| = (2r - h)a_h^r.$  
\end{proof}

 \begin{lemma}\label{vertex cnb}
If $0\leq h \leq 2r-1$  and  $Y$ is a subgraph of the graph $G$ with minimum degree at least $h,$ then $|N[Y]|\geq a_h^r(2r-h+1).$
   \end{lemma}
   \begin{proof}
   	We proceed by induction on $r.$ The result obviously holds for $r = 1$ and also for the case $ h = 0.$   Suppose $ r\geq 2$ and $ h \geq 1.$  Assume that the result holds for a graph that is the  product of $r-1$ cycles.  Let $G = C_{k_1}\Box C_{k_2}\Box \cdots \Box C_{k_{r}}.$ Then  $G=H 
   	\Box C_{k_r}$, where $H =C_{k_1}\Box C_{k_2}\Box \cdots \Box C_{k_{r-1}}.$ Then $G$ contains $k_r$ vertex-disjoint copies $H^1, H^2, \dots, H^{k_r}$ of $H.$ Then every vertex of $H^i$ has one neighbour in $H^{i-1}$ and $H^{i+1},$ where the addition and subtraction in the superscript is carried out modulo $k_r.$ Let $Y$ be a subgraph of  $G$ with $\delta(Y) \geq h.$ 
   If $Y$  spans $G,$ then, by $\delta(K),$  $|N[Y]| = |V(Y)| = k_1 k_2\dots k_r \geq a_h^r(2r-h+1).$  Therefore we may assume that $Y$ does  not span $G.$  Since the minimum degree of every connected component of  $Y$ is at least the minimum degree of $Y,$ we may assume that $Y$ is connected. 
   
   Note that  $Y$ intersects at least one copy of $H^i.$   Let $ Y_i = Y\cap H^i$ for $i=1,2,\dots,k_r.$  
 \vskip.2cm
 \noindent
   {\it Case 1.} $ Y_i \neq \emptyset$ for only one value of $i.$  
   \vskip.15cm
   \noindent
   We may assume that only $ Y_1 $ is non-empty. Then $Y = Y_1$ is contained in the graph $H^1.$  Since $H^1$  is $(2r -2)$-regular,  $h \leq 2r-2.$ Also, the minimum degree of $Y$ in $H^1$ is at least $h.$ We have $N[Y] = N_{H^1}[Y] \cup N_{H^{k_r}}(Y)\cup N_{H^{2}}(Y).$ If $ h = 2r-2,$ then $ Y = H^1$ and so, $N[Y]= V(H^1) \cup V(H^2)\cup V(H^{k_r})$ giving  $$|N[Y] |\geq 3|V(H^1)| = 3  k_1k_2\dots k_{r-1}\geq  12 k_1k_2\dots k_{r-2}= (2r-h+1)a_h^r.$$  Suppose  $0 \leq h \leq 2r-3 = 2(r-1)-1.$ Then, by induction, $|N_{H^1}[Y]|\geq a_h^{r-1}(2r-h-1).$ Also, by Lemma \ref{ahr 1}, $|N_{H^{k_r}}(Y)| =  |N_{H^{2}}(Y)| = |V(Y)| \geq a_h^{r-1}.$  Therefore, by Lemma \ref{pro ahr}(3), $$|N[Y]|\geq a_h^{r-1}(2r-h-1)+2a_h^{r-1} = (2r-h+1)a_h^{r-1}\geq  (2r-h+1)a_h^r.$$
 \vskip.1cm
  \noindent
 {\it Case 2.} $Y_i\neq \emptyset $ for all $i=1,2,\dots,k_r.$   
 \vskip.15cm
 \noindent
  In this case,  $N[Y] \supseteq N_{H^1}[Y_1]\cup N_{H^2}[Y_2]\cup \dots \cup N_{H^{k_r}}[Y_{k_r}].$ If  $h=1,$ then $\delta_{H^i}(Y_i) \geq  0$ and so, by induction, $|N_{H^i}[Y_i]|\geq a_{0}^{r-1}(2(r-1)-0+1)=2r-1$ giving 
 $$ |N[Y]|\geq |N_{H^1}[Y_1]|+ |N_{H^2}[Y_2]| + \dots + |N_{H^{k_r}}[Y_{k_r}]| \geq k_r (2r-1)\geq  8r-4\geq 4r \geq a_1^r(2r) =a_{h}^{r}(2r-h+1) .$$
 Suppose $h\geq 2$. Then  $\delta_{H^i}(Y_i) \geq  h-2$  and so, by  induction, $|N_{H^i}[Y_i]|\geq a_{h-2}^{r-1}(2r-h+1)$ for all $i.$  Therefore, by Lemma \ref{pro ahr}(2),
 $$ |N[Y]|\geq |N_{H^1}[Y_1]|+ |N_{H^2}[Y_2]| + \dots + |N_{H^{k_r}}[Y_{k_r}]| \geq k_r a_{h-2}^{r-1}(2r-h+1) \geq a_{h}^{r}(2r-h+1) .$$
 \vskip.1cm
  \noindent
  {\it Case 3.} $Y_i\neq \emptyset $ for more than one but not all values of $i.$ 
  \vskip.15cm
  \noindent  
 We may assume that $ Y_1 $ is non-empty but $Y_{k_r}$ is empty. Let $ t$ be the largest integer such that $Y_t$ is non-empty. Then $ 1 < t < k_r$ ; see Figure 3. The minimum degree of $Y_i$ in $H^i$ is at least $h-1$ for $ i = 1, t$ and it is at least $h - 2$ for $ 2 < i < t.$ 
 Suppose that $h=2r-1.$ Then $Y_1 = H^1$  and $Y_t = H^t.$ Hence  $N[Y]\supseteq V(H^1)\cup V(H^2) \cup V(H^t) \cup V(H^{t+1})\cup  V(H^{k_r}).$ Since $ k_r \geq 4,$ $ t \neq 2$ or $ t+1 \neq k_r$ and $|V(H^1)| = |V(H^i)| $ for all $i > 1.$  By Lemma \ref{spanning}, $$|N[Y]|\geq 4|V(H^1)|= 4|V(H)| \geq 4k_1k_2\dots k_{r-1} \geq k_1k_2\dots k_{r-1} \geq (2r-h+1)a_h^r.$$

\begin{center}
\unitlength 1.3mm 
\linethickness{0.4pt}
\ifx\plotpoint\undefined\newsavebox{\plotpoint}\fi 
\begin{picture}(74.5,58.75)(0,0)
\put(-.75,8){\framebox(9.5,26.25)[cc]{}}
\put(31.75,8){\framebox(9.5,26.25)[cc]{}}
\put(14.75,8){\framebox(9.5,26.25)[cc]{}}
\put(47.25,8){\framebox(9.5,26.25)[cc]{}}
\put(65,8){\framebox(9.5,26.25)[cc]{}}
\put(8.75,29.75){\line(1,0){6}}
\put(41.25,29.75){\line(1,0){6}}
\put(9,25.25){\line(1,0){5.75}}
\put(41.5,25.25){\line(1,0){5.75}}
\put(8.75,15.25){\line(1,0){6}}
\put(41.25,15.25){\line(1,0){6}}
\put(12,21.5){\makebox(0,0)[cc]{$\vdots$}}
\put(27.75,21.5){\makebox(0,0)[cc]{$\vdots$}}
\put(60.75,21.5){\makebox(0,0)[cc]{$\vdots$}}
\put(44.5,21.5){\makebox(0,0)[cc]{$\vdots$}}
\qbezier(0,34.25)(34,58.75)(73,34.25)
\qbezier(3,34.5)(32.875,52)(69.25,34.5)
\qbezier(8.25,34.25)(35.25,42.375)(65.25,34)
\put(36.5,41.25){\makebox(0,0)[cc]{$\vdots$}}
\put(3.75,21.75){\oval(5.5,10.5)[]}
\put(19.75,22){\oval(5.5,10.5)[]}
\put(36.75,22){\oval(5.5,10.5)[]}

\put(24.696,14.946){\line(1,0){.9688}}
\put(26.633,14.946){\line(1,0){.9688}}
\put(28.571,14.946){\line(1,0){.9688}}
\put(30.508,14.946){\line(1,0){.9688}}
\put(24.196,24.946){\line(1,0){.9688}}
\put(26.133,24.946){\line(1,0){.9688}}
\put(28.071,24.946){\line(1,0){.9688}}
\put(30.008,24.946){\line(1,0){.9688}}
\put(24.446,29.946){\line(1,0){.9688}}
\put(26.383,29.946){\line(1,0){.9688}}
\put(28.321,29.946){\line(1,0){.9688}}
\put(30.258,29.946){\line(1,0){.9688}}
\put(57.446,15.696){\line(1,0){.9688}}
\put(59.383,15.696){\line(1,0){.9688}}
\put(61.321,15.696){\line(1,0){.9688}}
\put(63.258,15.696){\line(1,0){.9688}}
\put(56.946,25.696){\line(1,0){.9688}}
\put(58.883,25.696){\line(1,0){.9688}}
\put(60.821,25.696){\line(1,0){.9688}}
\put(62.758,25.696){\line(1,0){.9688}}
\put(57.196,30.696){\line(1,0){.9688}}
\put(59.133,30.696){\line(1,0){.9688}}
\put(61.071,30.696){\line(1,0){.9688}}
\put(63.008,30.696){\line(1,0){.9688}}
\put(4,21){\makebox(0,0)[cc]{$Y_1$}}
\put(19.75,22){\makebox(0,0)[cc]{$Y_2$}}
\put(37,21.75){\makebox(0,0)[cc]{$Y_t$}}
\put(3.25,5){\makebox(0,0)[cc]{$H^1$}}
\put(19,5.25){\makebox(0,0)[cc]{$H^2$}}
\put(36,5){\makebox(0,0)[cc]{$H^t$}}
\put(51.25,5){\makebox(0,0)[cc]{$H^{t+1}$}}
\put(69.5,5.5){\makebox(0,0)[cc]{$H^{k_r}$}}
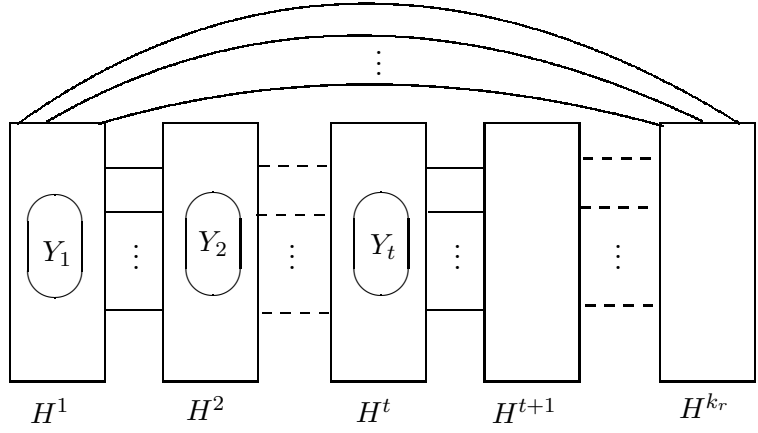
\put(36.25,-2){\makebox(0,0)[cc]{Figure 3. The graph $G$ with $Y_j=\emptyset$ for $t<j\leq k_r$ }}
\end{picture}

 \end{center}

\vskip.3cm
 Suppose that $0\leq h \leq 2r-2.$  The graph $Y_i$ has $|V(Y_i)|$ neighbours in $H^{i-1}$ and $H^{i+1}$ for $ i = 1, t.$ Therefore  $|N[Y]| \geq |N_{H^1}[Y_1] | + |N_{H^t}[Y_{t}]| + |V(Y_1)| + |V(Y_t)|.$   If $  i \in \{1, t\}, $ then $\delta_{H^i}(Y_i)\geq h-1$ and so, by induction, $|N_{H^{i}}[Y_i]|\geq a_{h-1}^{r-1}(2r-h).$ Also, by Lemma \ref{ahr 1}, $|V(Y_i)|\geq a_{h-1}^{r-1}.$ Hence, by Lemma \ref{pro ahr}(1), we have 
 $$|N[Y]| \geq  2a_{h-1}^{r-1}(2r-h)+2a_{h-1}^{r-1} = a_{h}^{r}(2r-h)+a_{h}^{r} = a_{h}^{r}(2r-h+1).$$

 Thus $|N[Y]|\geq a_{h}^{r}(2r-h+1)$ in each case. This completes the proof.    \end{proof}
\begin{proposition}
 If $0\leq h \leq 2r-2$ and $S$ is a  conditional $h$-edge cut of the graph $G, $ then $|S| \geq a_h^r(2r-h).$  
 \end{proposition}
\begin{proof}
We argue  by induction on $r.$ The result follows easily for $r=1$ and so for $h=0.$  Suppose $r\geq 2$ and $h \geq 1.$ Assume that the  result holds true for the Cartesian product of $r-1$ cycles, each of length at least 4. Let  $ G = C_{k_1} \Box C_{k_2} \Box \dots \Box C_{k_r}.$   Then $G = H \Box C_{k_r},$ where $ H =C_{k_1} \Box C_{k_2} \Box \dots \Box C_{k_{r-1}}.$ Then $G$ is obtained by replacing $i$th vertex of $C_{k_r}$ by the copy $H^i$ of $H$ and replacing each edge $C_{k_r}$  by the  matching between the two copies of $H^i$ corresponding to the end vertices of that edge.     

As $S$ is a  $h$-vertex cut of $G,$ the graph $ G - S$ is disconnected and each component of it has minimum degree $h.$ Let $Y$ be a subgraph of $G-S$ consisting of at least one but not all components of $G-S$ and let $Z$ be the subgraph consisting of the remaining components. Thus  $G-S=Y\cup Z$ and further, $ \delta (Y)\geq h$ and $\delta(Z) \geq h.$  As $S$ is a cut, $N(Y)\subseteq S $ and $N(Z)\subseteq S $  and so,  $|S|\geq |N(Y)|$  and  $|S|\geq |N(Z)|.$  Note that $Y$ and $Z$ each intersects  $H^i$  for at least one $i.$  Let $S_i=S\cap V(H^{i}),$ $Y_i=Y\cap V(H^{i})$ and $Z_i=Z\cap V(H^{i}).$  Depending upon the nature of $Y$ and $Z$, the  proof is divided into several cases.
\vskip.2cm
 \noindent
 {\it Case 1.} Suppose $Y_i\neq \emptyset$ for only one $i.$
 \vskip.15cm
 \noindent
  We may assume that only $Y_1$ is non-empty. Then $Y=Y_1$ is contained in $H^1.$ Therefore the minimum degree of $Y_1$ in $H^1$ is at least $h.$ As $H^1$ is $(2r-2)$-regular, $0\leq h \leq 2r-2.$  Suppose $h=2r-2.$ Then $Y=H^1$ and  $N(Y)=V(H^{k_r})\cup V(H^2).$ 
 Therefore
  $$ |S|\geq |N(Y)|=  |V(H^{k_r})|+|V(H^{2})| = 2k_1k_2\dots k_{r-2}k_{r-1}\geq 8k_1k_2\dots k_{r-2}= a_{h}^{r}(2r-h).$$
  
 Suppose that $0\leq h \leq 2r-3=2(r-1)-1.$  The graph $Y$ has $|V(Y)|$ neighbours in each of $H^{k_r}$ and $H^{2}.$  Therefore  
$|N(Y)| =|N_{H^1}(Y)|+ |V(Y)|+ |V(Y)|= |N_{H^1}[Y]|+ |V(Y)|. $ By Lemmas \ref{pro ahr}(3), \ref {ahr 1} and \ref{vertex cnb},  
$$ |S|\geq |N(Y)|\geq a_{h}^{r-1}(2r-h-1)+  a_{h}^{r-1}=  a_{h}^{r-1}(2r-h)\geq a_{h}^{r}(2r-h).$$
   \noindent
  \textit{Case 2.} Suppose $Y_i\neq \emptyset$ for more than one but not all values of $i.$
  \vskip.15cm
  \noindent
    We may assume that $Y_1$ is non empty but $Y_{k_r}$ is empty.  Let $t$ be the largest integer such that $Y_t$ is non-empty. Then $1<t<k_r.$ The minimum degree of $Y_i$ in $H^i$ is at least $h-1$ for $i =1,~t$ and it is at least $h-2$ for $2<i<t.$ We consider the following subcases. 
 \vskip.15cm 
\indent  \textit{Case 2.1.} Suppose $Y_i=\emptyset$ for more than one value of $i.$ 

In this case  $t<k_r-1.$ The graph $Y_1$ has $|V(Y_1)|$ neighbours in $H^{K_r}.$ Similarly, the graph $Y_t$ has $|V(Y_t)|$ neighbours in $H^{t+1}.$ Hence $N(Y)\supseteq N_{H^1}(Y_1)\cup V(Y_1)\cup N_{H^t}(Y_t)\cup V(Y_t).$ Therefore, by Lemmas \ref{pro ahr}(1) and \ref{vertex cnb},
$$ |N(Y)|\geq  |N_{H^1}(Y_1)|+ |V(Y_1)|+|N_{H^t}(Y_t)|+|V(Y_t)| \geq |N_{H^1}[Y_1]|+|N_{H^t}[Y_t]|\geq (2r-h)a_{h-1}^{r-1}\geq  a_{h}^{r}(2r-h).$$
 \indent  \textit{Case 2.2.} Suppose $Y_i=\emptyset$ for exactly one $i.$ 
 
Suppose $t=k_r-1.$ Here we calculate $|S_i|$ by using Lemma \ref{vertex cnb} or induction. To use induction, we need to consider the nature of the graph $Z$ also.  If $Z_i\neq \emptyset $ for only one value of $i,$ then result follows from Case 1. If $Z_i=\emptyset $ for more than one value of $i,$ then the result follows from Case 2.1. It remains to consider the following cases about $Z.$
   \vskip.15cm
      \indent \textit{Subcase 1.}  $Z_i=\emptyset$ for exactly one value of $i.$  
      We have again the following two subcases.
         
         \indent \textit{(i).} Suppose $Z_{k_r}=\emptyset.$ Recall that $Y_{k_r}=\emptyset; $ see Figure 4(a)
The graph $Y_1$ has $|V(Y_1)|$ neighbours in $H^{k_r}.$  Hence $S_{k_r}$ must contains $|V(Y_1)|$ vertices.  Thus $S_1\cup S_{k_r}\supseteq N_{H^1}(Y_1)\cup V(Y_1)\supseteq N_{H^1}[Y_1]$ and by Lemma \ref{vertex cnb}, $|N_{H^1}[Y_1]|\geq (2r-h)a_{h-1}^{r-1}.$
 Hence, $|S_1\cup S_{k_r}|\geq (2r-h)a_{h-1}^{r-1}.$

Suppose $h=1$. Then $\delta_{H^i}(Y_i)\geq 0$ and $\delta_{H^i}(Z_i)\geq 0$ for all $i.$ Hence $|S_1\cup S_{k_r}|\geq (2r-1)a_0^{r-1}=(2r-1).$ By induction, $|S_i|\geq (2(r-1)-0)a_0^{r-1}=2r-2$ for $i\in \{2, 3, \dots k_r-1\}.$ Note that  $S$ contains $S_1\cup S_2\cup \dots S_{k_r-2}\cup S_{k_r-1}.$ Therefore
\begin{eqnarray*}
|S|& \geq &  (|S_1 \cup S_{k_r}|)+ \sum\limits_{i=2}^{i=k_r-1} |S_i| \\
 &\geq & (2r-1)+ \sum\limits_{i=2}^{i=k_r-1}(2r-2) \\
 &=& (2r-1)+(2r-2)(k_r-2)\\
	 & \geq & 2(2r-1) \\
	 &= & (2r-h)a_h^r. 
\end{eqnarray*}
By our assumption, both $Y_{k_r-1}$ and $Z_{k_r-1}$ are non-empty. Note that $\delta_{H_{k_r-1}}(Y_{k_r-1}) \geq h-1$ and $\delta_{H_{k_r-1}}(Z_{k_r-1}) \geq h-1$. As $h-1>h-2,$ we can say that $\delta_{H_{k_r-1}}(Y_{k_r-1}) \geq h-2$ and $\delta_{H_{k_r-1}}(Z_{k_r-1}) \geq h-2$. Thus $S_{k_r-1}$ is a $(h-2)$-cut in $H^{k_r-1}.$ Also $Y_i$ and $Z_i$ both are non-empty sets with minimum degree $h-2$ in $H^i$ for $i\in \{2, 3, \dots k_r-2\}.$ 

Suppose $h\geq 2$. Then $S_i$ is an $(h-2)$-cut in $H^i$  and so, by induction, $|S_i|\geq (2r-h)a_{h-2}^{r-1}$ for $i\in \{2, 3, \dots k_r-1  \}.$
As $S$ contains $S_1\cup S_1\cup \dots S_{k_r-2}\cup S_{k_r-1},$ we have
 \begin{eqnarray*}
|S| &\geq & (|S_1 \cup S_{k_r}|)+ \sum\limits_{i=2}^{i=k_r-1} |S_i|\\
    &\geq & (2r-h)a_{h-1}^{r-1}+ \sum\limits_{i=2}^{i=k_r-1}(2r-h)a_{h-2}^{r-1}\\
    & = & (2r-h)a_{h-1}^{r-1}+(k_r-2)(2r-h)a_{h-2}^{r-1}\\
    & \geq & (2r-h)a_{h-1}^{r-1}+ \frac{k_r}{2} a_{h-2}^{r-1}(2r-h)..............\rm{(since~~ k_r~~\geq~~ 4)}\\
    & \geq & (2r-h)a_{h-1}^{r-1}+ \frac{1}{2} a_{h}^{r}(2r-h)....................\rm{(by ~~Lemma~~\ref{pro ahr} (2))}\\
    & = & (2r-h)a_{h-1}^{r-1}+ a_{h-1}^{r-1}(2r-h)....................\rm{(by~~Lemma~~\ref{pro ahr} (1))}\\
    & =& 2a_{h-1}^{r-1}(2r-h)\\
    & = & a_{h}^{r}(2r-h)................................................\rm{(by~~Lemma~~\ref{pro ahr}(1))}
\end{eqnarray*}

\begin{center} 
\unitlength 1mm 
\linethickness{0.4pt}
\ifx\plotpoint\undefined\newsavebox{\plotpoint}\fi 
\begin{picture}(159.393,58.75)(0,0)
\put(.75,8){\framebox(9.5,26.25)[cc]{}}
\put(84.143,8){\framebox(9.5,26.25)[cc]{}}
\put(33.25,8){\framebox(9.5,26.25)[cc]{}}
\put(116.643,8){\framebox(9.5,26.25)[cc]{}}
\put(16.25,8){\framebox(9.5,26.25)[cc]{}}
\put(99.643,8){\framebox(9.5,26.25)[cc]{}}
\put(48.75,8){\framebox(9.5,26.25)[cc]{}}
\put(132.143,8){\framebox(9.5,26.25)[cc]{}}
\put(66.5,8){\framebox(9.5,26.25)[cc]{}}
\put(149.893,8){\framebox(9.5,26.25)[cc]{}}
\put(10.25,29.75){\line(1,0){6}}
\put(93.643,29.75){\line(1,0){6}}
\put(42.75,30){\line(1,0){6}}
\put(126.143,30){\line(1,0){6}}
\put(10.25,26){\line(1,0){6}}
\put(93.643,26){\line(1,0){6}}
\put(42.75,26.25){\line(1,0){6}}
\put(126.143,26.25){\line(1,0){6}}
\put(10.5,13.5){\line(1,0){6}}
\put(93.893,13.5){\line(1,0){6}}
\put(43,13.75){\line(1,0){6}}
\put(126.393,13.75){\line(1,0){6}}
\put(13.5,21.5){\makebox(0,0)[cc]{$\vdots$}}
\put(96.893,21.5){\makebox(0,0)[cc]{$\vdots$}}
\put(29.75,21.5){\makebox(0,0)[cc]{$\vdots$}}
\put(113.143,21.5){\makebox(0,0)[cc]{$\vdots$}}
\put(62.25,21.75){\makebox(0,0)[cc]{$\vdots$}}
\put(145.643,21.75){\makebox(0,0)[cc]{$\vdots$}}
\put(46,21.5){\makebox(0,0)[cc]{$\vdots$}}
\put(129.393,21.5){\makebox(0,0)[cc]{$\vdots$}}
\qbezier(1.5,34.25)(35.5,58.75)(74.5,34.25)
\qbezier(84.893,34.25)(118.893,58.75)(157.893,34.25)
\qbezier(4.5,34.5)(34.375,52)(70.75,34.5)
\qbezier(87.893,34.5)(117.768,52)(154.143,34.5)
\qbezier(9.75,34.25)(36.75,42.375)(66.75,34)
\qbezier(93.143,34.25)(120.143,42.375)(150.143,34)
\put(37.5,41.5){\makebox(0,0)[cc]{$\vdots$}}
\put(120.893,41.5){\makebox(0,0)[cc]{$\vdots$}}
\put(5.25,27.25){\oval(5.5,10.5)[]}
\put(88.643,27.25){\oval(5.5,10.5)[]}
\put(20.75,27.25){\oval(5.5,10.5)[]}
\put(104.143,27.25){\oval(5.5,10.5)[]}
\put(37.5,27.25){\oval(5.5,10.5)[]}
\put(37.5,13.723){\oval(5.5,10.5)[]}
\put(120.893,27.25){\oval(5.5,10.5)[]}
\put(53.5,27.25){\oval(5.5,10.5)[]}
\put(136.893,27.25){\oval(5.5,10.5)[]}
\put(5,13.75){\oval(5.5,10.5)[]}
\put(88.393,13.75){\oval(5.5,10.5)[]}
\put(20.5,13.75){\oval(5.5,10.5)[]}
\put(103.893,13.75){\oval(5.5,10.5)[]}
\put(53.25,13.75){\oval(5.5,10.5)[]}
\put(136.643,13.75){\oval(5.5,10.5)[]}
\put(154.143,14){\oval(5.5,10.5)[]}
\put(33.446,30.196){\line(0,1){0}}
\put(116.839,30.196){\line(0,1){0}}
\put(26.196,29.946){\line(1,0){.9688}}
\put(28.133,30.008){\line(1,0){.9688}}
\put(30.071,30.071){\line(1,0){.9688}}
\put(32.008,30.133){\line(1,0){.9688}}
\put(109.589,29.946){\line(1,0){.9688}}
\put(111.527,30.008){\line(1,0){.9688}}
\put(113.464,30.071){\line(1,0){.9688}}
\put(115.402,30.133){\line(1,0){.9688}}
\put(58.696,30.196){\line(1,0){.9688}}
\put(60.633,30.258){\line(1,0){.9688}}
\put(62.571,30.321){\line(1,0){.9688}}
\put(64.508,30.383){\line(1,0){.9688}}
\put(142.089,30.196){\line(1,0){.9688}}
\put(144.027,30.258){\line(1,0){.9688}}
\put(145.964,30.321){\line(1,0){.9688}}
\put(147.902,30.383){\line(1,0){.9688}}
\put(26.196,25.946){\line(1,0){.9688}}
\put(28.133,26.008){\line(1,0){.9688}}
\put(30.071,26.071){\line(1,0){.9688}}
\put(32.008,26.133){\line(1,0){.9688}}
\put(109.589,25.946){\line(1,0){.9688}}
\put(111.527,26.008){\line(1,0){.9688}}
\put(113.464,26.071){\line(1,0){.9688}}
\put(115.402,26.133){\line(1,0){.9688}}
\put(58.696,26.196){\line(1,0){.9688}}
\put(60.633,26.258){\line(1,0){.9688}}
\put(62.571,26.321){\line(1,0){.9688}}
\put(64.508,26.383){\line(1,0){.9688}}
\put(142.089,26.196){\line(1,0){.9688}}
\put(144.027,26.258){\line(1,0){.9688}}
\put(145.964,26.321){\line(1,0){.9688}}
\put(147.902,26.383){\line(1,0){.9688}}
\put(26.196,13.696){\line(1,0){.9688}}
\put(28.133,13.758){\line(1,0){.9688}}
\put(30.071,13.821){\line(1,0){.9688}}
\put(32.008,13.883){\line(1,0){.9688}}
\put(109.589,13.696){\line(1,0){.9688}}
\put(111.527,13.758){\line(1,0){.9688}}
\put(113.464,13.821){\line(1,0){.9688}}
\put(115.402,13.883){\line(1,0){.9688}}
\put(58.696,13.946){\line(1,0){.9688}}
\put(60.633,14.008){\line(1,0){.9688}}
\put(62.571,14.071){\line(1,0){.9688}}
\put(64.508,14.133){\line(1,0){.9688}}
\put(142.089,13.946){\line(1,0){.9688}}
\put(144.027,14.008){\line(1,0){.9688}}
\put(145.964,14.071){\line(1,0){.9688}}
\put(147.902,14.133){\line(1,0){.9688}}
\put(33.446,30.196){\line(0,1){0}}
\put(116.839,30.196){\line(0,1){0}}
\put(5.25,27.25){\makebox(0,0)[cc]{\tiny $Y_1$}}
\put(88.643,27.25){\makebox(0,0)[cc]{\tiny $Y_1$}}
\put(21,27.25){\makebox(0,0)[cc]{\tiny $Y_2$}}
\put(104.393,27.25){\makebox(0,0)[cc]{\tiny $Y_2$}}
\put(37.75,27.25){\makebox(0,0)[cc]{\tiny $Y_i$}}
\put(37.75,13.723){\makebox(0,0)[cc]{\tiny $Z_i$}}
\put(121.143,27.25){\makebox(0,0)[cc]{\tiny $Y_i$}}
\put(53.75,27.25){\makebox(0,0)[cc]{\tiny $ Y_{i+1}$}}
\put(137.143,27.25){\makebox(0,0)[cc]{\tiny $ Y_{i+1}$}}
\put(5,13.75){\makebox(0,0)[cc]{\tiny $Z_1$}}
\put(88.393,13.75){\makebox(0,0)[cc]{\tiny $Z_1$}}
\put(20.75,13.75){\makebox(0,0)[cc]{\tiny $Z_2$}}
\put(104.143,13.75){\makebox(0,0)[cc]{\tiny $Z_2$}}
\put(53.5,13.75){\makebox(0,0)[cc]{\tiny $Z_{i+1}$}}
\put(136.893,13.75){\makebox(0,0)[cc]{\tiny $Z_{i+1}$}}
\put(154.143,14){\makebox(0,0)[cc]{\tiny $Z_{k_r}$}}
\put(4.75,5){\makebox(0,0)[cc]{ $H^1$}}
\put(88.143,5){\makebox(0,0)[cc]{ $H^1$}}
\put(20.5,5.25){\makebox(0,0)[cc]{ $H^2$}}
\put(103.893,5.25){\makebox(0,0)[cc]{ $H^2$}}
\put(37.5,5){\makebox(0,0)[cc]{ $H^i$}}
\put(120.893,5){\makebox(0,0)[cc]{ $H^i$}}
\put(52.75,5){\makebox(0,0)[cc]{ $H^{i+1}$}}
\put(136.143,5){\makebox(0,0)[cc]{ $H^{i+1}$}}
\put(71,5.5){\makebox(0,0)[cc]{ $H_{k_r}$}}
\put(154.393,5.5){\makebox(0,0)[cc]{ $H_{k_r}$}}
\put(37.75,-2){\makebox(0,0)[cc]{ (a). $Z_{k_r}=\emptyset$}}
\put(121.143,-2){\makebox(0,0)[cc]{(b). $Z_i=\emptyset$ for some $ i < k_r$}}
\put(80,-8){\makebox(0,0)[cc]{Figure 4. The graph $G$ with $Y_{k_r}=\emptyset$}}
\end{picture}
\end{center}

\vskip.8cm

 \indent \textit{(ii).} Suppose  $Z_{k_r}$ is non-empty. 
 
 Then $Z_i$ is empty for some $ i < k_r;$ see Figure 4(b).  Note that the  minimum degree of $Y_1$ and $Z_{i+1}$ is at least $h-1$ in $H^1$ and $H^{i+1}$, respectively. Observe that $S_1\supseteq N_{H^1}(Y_1)$ and $S_{i+1} \supseteq N_{H^{i+1}}(Z_{i+1}).$ Also, each vertex of $Y_1$ has a neighbour in $H^{k_r}$  and no neighbour in $Z_{k_r}.$ Hence $N_{H^{k_r}}(Y_1)$ must be in $S_{k_r}.$ Similarly, $S_i$ contains $N_{H^i}(Z_{i+1}).$ Thus $S_1\cup S_{k_r}\supseteq N_{H^1}(Y_1)\cup N_{H^{k_r}}(Y_1)=N_{H^1}(Y_1)\cup V(Y_1)=N_{H^1}[Y_1]$ and $S_i\cup S_{i+1}\supseteq N_{H^i}(Z_{i+1}) \cup N_{H^{i+1}}(Z_{i+1})= N_{H^{i+1}}(Z_{i+1})\cup V(Z_{i+1})=N_{H^{i+1}}[Z_{i+1}].$ By Lemma \ref{vertex cnb}, $|N_{H^j}[Y_j]|\geq a_{h-1}^{r-1}(2r-h)$ for $j=1,~i+1.$
 Hence $|S_1\cup S_{k_r}|=|S_i\cup S_{i+1}|\geq a_{h-1}^{r-1}(2r-h).$  By Lemma \ref{pro ahr}(1), 
 $$ |S| \geq |S_1\cup S_{k_r}|+ |S_i\cup S_{i+1}|=2a_{h-1}^{r-1}(2r-h)=a_h^r (2r-h).$$

  \indent \textit{ Subcase 2.}  Suppose that $Z_i \neq \emptyset$ for $i=1,~2, \dots,k_r.$    
  
   By the arguments similar to that in Subcase 1(i), one can easily show that $|S_1\cup S_{k_r}|\geq a_{h-1}^{r-1}(2r-h)$ and $|S_i|\geq (2r-h)a_{h-2}^{r-1}$ for $i\in \{2,3,\dots k_{r-1}\}.$  Thus, in this case also, $|S| \geq a_h^r (2r-h).$
  \vskip.2cm
   \noindent  
       \textit{Case 3.} Suppose $Y_i \neq \emptyset$  for $i=1,~2, \dots,k_r.$
   \vskip.15cm
      \noindent     
    If $Z$ does not intersect  $H^i$ for some $i,$ then the result follows by replacing $Y$ by $Z$ in Case 1 and Case 2. Suppose that $Z$ intersects  $H^i$ for all $i=1,~2, \dots, k_r.$  If $h=1,$ then the minimum degree of $Y_i$ and $Z_i$ is at least $0$ and so, by induction, $|S_i|\geq a_{0}^{r-1}(2(r-1)-0)=2r-2$ giving
        $$|S|  =  \sum\limits_{i=1}^{i=k_r}|S_i|\geq  \sum\limits_{i=1}^{i=k_r}(2r-2)\geq k_r (2r-2) \geq 4(2r-2)=8(2r-1)> 2(2r-1)=a_1^r (2r-1).$$
       Suppose $h\geq 2$.  The minimum degree of $Y_i$ and $Z_i$ is at least $h-2.$ This shows that  $S_i$ is a conditional  $(h-2)$-vertex cut of the graph $H^i$ for $i=1,~2, \dots,k_r.$ Therefore, by induction and by Lemma \ref{pro ahr}(2), we have  
 $$|S|  =  \sum\limits_{i=1}^{i=k_r}|S_i|\geq  \sum\limits_{i=1}^{i=k_r}a_{h-2}^{r-1}(2r-h)\geq k_ra_{h-2}^{r-1}(2r-h) \geq a_h^r(2r-h).$$
Thus  $|S| \geq  a_h^r (2r-h)$ in all the above cases. This complete the proof. \end{proof}
\begin{corollary}\label{vertex lb}
 $\kappa^h(G) \geq a_h^r(2r-h).$ 
\end{corollary}
 It follows from Corollary \ref{vertex lb} and Lemma \ref{vertex ub} that $\kappa^h(G) = a_h^r(2r-h)$ for the graph $G$ of Main Theorem \ref{main theorem}.

\section{Conditional Edge Connectivity}

In this section, we prove that the conditional edge connectivity of the graph $G$ of   Main Theorem \ref{main theorem} is same as the conditional vertex connectivity of $G.$ 

Recall that  $G = C_{k_1} \Box C_{k_2} \Box \dots \Box C_{k_r}$ with $ 4 \leq k_1 \leq k_2\leq \dots \leq k_{r}$  and $W_h^r$ is a $h$-regular subgraph of $G$ with  $a_h^r$ vertices. For a  subgraph $K$ of  $G,$ let  
$$E_r(K)=\{xy: x\in V(K) ~~\rm{and} ~~y \in V(G) -V(K) \}.$$
\begin{lemma}\label{edge ub}
For $0\leq h \leq 2r-1,$ $\lambda^h(G)\leq (2r-h)a_h^r.$
\end{lemma}
\begin{proof}
Let $K=W_h^r.$  Then $K$ is  $h$-regular and  $G$ is $2r$-regular. Hence $|E_r(K)|=(2r-h)|V(K)|$ and  $G-E_r(K)$ is disconnected with  $K$ is one of its components.  By Lemma \ref{1 nb},  the minimum degree of  every component of $G-E_r(K)$ other than $K$ is at least $2r-1\geq h.$  Therefore 
$E_r(K)$ contains an $h$-edge cut of $G$ giving  $\lambda^h(G)\leq |E_r(K)|=(2r-h)a_h^r.$
\end{proof}
\begin{lemma}\label{edge cnb} For a subgraph  $Y$  of $G$  of minimum degree at least $h,$ $|V(Y)|+|E_r(Y)|\geq a_h^r(2r-h+1).$
   \end{lemma}
\begin{proof}
If $Y$ is a spanning subgraph of $G,$ then  the edge set $E_r(Y)$ is  empty  and hence, by Lemma \ref{spanning}, $ |V(Y)|+|E_r(Y)|= k_1 k_2\dots k_r \geq a_h^r(2r-h+1).$  Suppose $Y$ is not a spanning subgraph of $G.$  Given a vertex $x$ in the neighbourhood $ N(Y)$ of $Y,$ there is at least one vertex $y\in Y$ which is adjacent to $x.$ Hence $xy\in E_r(Y).$ This implies that  $|N(Y)|\leq |E_r(Y)|.$ Hence, by Lemma \ref{vertex cnb},  $|Y|+|E_r(Y)|\geq |Y|+|N(Y)|=|N[Y]|\geq a_{h}^{r}(2r-h+1).$ 
\end{proof}
\begin{proposition}
 Let $F$ be a conditional $h$-edge cut of the graph $G.$ Then $|F| \geq a_h^r(2r-h).$  
\end{proposition}
\begin{proof}
Since the minimum degree of the graph $G$ is $2r,$ $ 0\leq h \leq 2r -1.$  We proceed  by induction on $r.$ The result follows trivially for $r=1.$ Suppose $r\geq 2.$ Assume the result  for Cartesian product of $r-1$ cycles. Let $F$ be a conditional $h$-edge cut of $G.$ Then $G-F$ is disconnected and every component of it has minimum degree at least $h.$ 
 Let  $Y$ be a subgraph of  $G-F$ consisting  of at least one but not all components of $G-F$  and let $Z$ be the subgraph consisting of the remaining components. Then $Y$ and $Z$ are vertex dsijoint subgraphs of $G- F$ of minimum degree at least $h$ and their union is  $G-F.$ Note that $E_r(Y)\subseteq F.$ Hence $|F|\geq |E_r(Y)|.$ Similarly, $|F|\geq |E_r(Z)|.$ 
 
 Write $G$ as $H\Box C_{k_r},$ where $H = C_{k_1} \Box C_{k_2} \Box \dots \Box C_{k_r}.$ Then $G$ is obtained by replacing vertex $i$ of the cycle $C_{k_r}$ by a copy $H^i$ of $H$ and replacing the edge  joining $i$ and $ i + 1 \pmod {k_r} $ by the perfect matching $M_i$ between the corresponding vertices of $H^i$ and $H^{i+1\pmod {k_r}}.$ Then $Y$   intersects at least one $H^i$. Similarly, $Z$ intersects  at least one $H^i$. Let   $Y_i=Y\cap H^i$ and $Z_i=Z\cap H^i$ for $ i = 1, 2, \dots, k_r.$ 
 
 For a subgraph $K$ of $G,$ let  $M_i(K)$ be the set of all edges in the matching $M_i$ each having exactly one end vertex in $K.$
\vskip.2cm
 \noindent
  {\it Case 1.} Suppose $Y_i\neq \emptyset$ for only one value of $i.$
  \vskip.15cm
   \noindent
 We may assume that $Y_1$ is non-empty. Then $Y=Y_1$ and so $Y$  is contained in the graph $H^1.$ Hence the minimum degree of $Y$ in $H^1$ is at least $h.$  Since $H^1$ is $(2r-2)$-regular, $h\leq 2r-2.$ Suppose $h=2r-2.$ Then $Y=H^1.$ Therefore  $|E_r(Y)|=|M_1|+|M_{k_r}|=2|V(H^1)|=2k_1k_2\dots k_{r-1}.$  As $4\leq k_{r-1},$ we have $$a_h^r(2r-h)=2a_h^r=2.2^{r-(r-2)}k_1k_2\dots k_{r-2}=4k_1k_2\dots k_{r-2}\leq k_1k_2\dots k_{r-2}k_{r-1}<|E_r(Y)| \leq |F|.$$
  Suppose $h<2r-2.$ Then   $E_r(Y)\supseteq E_{r-1}(Y)\cup M_1(Y)\cup M_{k_r}(Y).$  As $|M_1(Y)|=|M_{k_r}(Y)|= |V(Y)|, $  by Lemmas \ref{pro ahr}(3), \ref{ahr 1} and  \ref{edge cnb}, we have 
  $$|E_r(Y)|\geq (|E_{r-1}(Y)| + |V(Y)|) + |V(Y)| \geq a_{h}^{r-1}(2r-h-1) + a_{h}^{r-1} =  a_{h}^{r-1}(2r-h) \geq a_{h}^{r}(2r-h).$$
  \vskip.1cm
 \noindent
  {\it Case 2.}  Suppose $Y_i\neq \emptyset$ for more than one but not all values of $i.$
   \vskip.15cm
     \noindent
  We may assume that $Y_1$ is non-empty but $Y_{k_r}$ is empty.  Let $t$ be the largest integer such that $Y_t$ is non-empty. Then $1<t<k_r.$ The minimum degree of $Y_i$ in $H^i$ is at least $h-1$ for $i =1,~t.$ The graph $Y_1$ has $|V(Y_1)|$  neighbours in $H^{k_r}$ and  $Y_t$ has $|V(Y_t)|$ neighbours in $H^{t+1}.$ Hence $E_r(Y)\supseteq E_{r-1}(Y_1)\cup E_{r-1}(Y_t)\cup M_{k_r}(Y_1)\cup M_t(Y_t).$ 
  
 Suppose $h=2r-1.$ Then $Y_j=H^j$ for $j=1,~t$  giving $M_{k_r}(Y_1)=  M_{k_r}(H^1) = M_{k_r}$ and $M_t(Y_t)= M_t(H^t) = M_t.$  Hence 
  $$a_h^r(2r-h)=a_h^r= 2k_1k_2\dots k_{r-1}= |V(H^1)| + |V(H^t)|  = |M_{k_r} | + | M_t| \leq |E_r(Y)| \leq |F|.$$
    Suppose $h\leq 2r-2.$ Then $h-1\leq 2r-3$ and so, by Lemmas \ref{edge cnb} and \ref{pro ahr}(1), 
    $$|F|\geq|E_r(Y)|\geq (|E_{r-1}(Y_1)|+|V(Y_1)|)+(|E_{r-1}(Y_t)|+|V(Y_t)|) \geq  2a_{h-1}^{r-1}(2r-h)=(2r-h)a_h^r.$$  
\vskip.1cm
 \noindent
  {\it Case 3.}
   Suppose $Y_i \neq \emptyset$ for all  $i=1,~2, \dots,k_r.$
    \vskip.15cm
      \noindent
      If the graph $Z$ does not intersect $H^i$ for some $i,$  then the result follows easily by replacing $Y$ by $Z$ in Case 1 and Case 2. Suppose  $Z$ intersects  $H^i$ for all $i=1,~2, \dots,k_r.$        
  If $h=1,$ then the minimum degree of $Y_i$ and $Z_i$ is at least 0 and so, by induction we have  $$|E_r(Y)|  =  \sum\limits_{i=1}^{i=k_r}|E_{r-1}(Y_i)|\geq  \sum\limits_{i=1}^{i=k_r}(2r-2)\geq k_r (2r-2) \geq 4(2r-2)=8(2r-1)> 2(2r-1)=a_1^r (2r-1).$$    Suppose $h\geq 2.$  The minimum degree of $Y_i$ and $Z_i$ is at least $h-2.$  Therefore the edge set $E_{r-1}(Y_i)$ is a conditional $(h-2)$-edge cut of $H^i.$ By induction, $|E_{r-1}(Y_i)|\geq a_{h-2}^{r-1}(2r-h)$ for $i=1,~2, \dots,k_r.$ By Lemma \ref{pro ahr}(2),
      $$|F|\geq |E_r(Y)|  =  \sum\limits_{i=1}^{i=k_r}|E_{r-1}(Y_i)|\geq  \sum\limits_{i=1}^{i=k_r}a_{h-2}^{r-1}(2r-h)\geq k_ra_{h-2}^{r-1}(2r-h) \geq a_h^r(2r-h).$$
This completes the proof.
\end{proof}
\begin{corollary}\label{edge lb}
 $\lambda^h(G) \geq a_h^r(2r-h).$ 
\end{corollary}


It follows from Corollary \ref{edge lb} and Lemma \ref{edge ub} that  $\lambda^h(G) = a_h^r(2r-h),$ where $G$ is the graph of Main Theorem \ref{main theorem}.  Thus $\kappa(G) = a_h^r(2r-h) = \lambda^h(G).$ \textbf{This completes the proof of Main Theorem \ref{main theorem}.
}

It is worth to mention that the edge connectivity part of Main Theorem \ref{main theorem} proves that the following conjecture of Xu \cite{xu} holds for the classes multidimensional tori and $k$-array $r$-cubes. 
\begin{conjecture}
 Let $k,~h$ be an integers and $G$ be a connected graph with minimum degree at least $k$ and $a_h(G)$ be the minimum cardinality of a vertex set of an $h$-regular subgraph of $G.$ If $\lambda^h(G)$ exists,  then  $\lambda^h(G) \leq a_h(G)(k-h).$ 
\end{conjecture}
\vskip.2cm\noindent
\textbf{Concluding Remarks:} 
\vskip.15cm\noindent
We determined the conditional $h$-vertex connectivity and the conditional $h$-edge connectivity of a Multidimensional Torus $G$ which is  the Cartesian product of $r$ cycles each of length at least four, for all possible values of $h.$  We first explored properties of a  $h$-regular subgraph of $G$  of smallest size and then established that  both these conditional conectivites are equal to $(2r - h)$ times the size of this subgraph. 


\begin{center}
{\bf Acknowledgement} 
\end{center}
The second author is financially supported by DST-SERB, Government of India through the project MTR/2018/000447.

\end{document}